\documentclass[11pt]{article}
\usepackage[margin=1in]{geometry}
\usepackage[T1]{fontenc}
\usepackage{lmodern,microtype,amsmath,amssymb,amsthm,mathtools}
\usepackage[authoryear,round]{natbib}
\usepackage[colorlinks=true,allcolors=blue]{hyperref}
\newtheorem{theorem}{Theorem}
\newtheorem{lemma}[theorem]{Lemma}
\newtheorem{corollary}[theorem]{Corollary}
\theoremstyle{definition}
\newtheorem{definition}[theorem]{Definition}
\newenvironment{proofof}[1]{\begin{proof}[Proof of #1]}{\end{proof}}
\newcommand{\E}{\mathbb E}
\newcommand{\R}{\mathbb R}

\newcommand{\elic}{\operatorname{elic}}
\newcommand{\iden}{\operatorname{iden}}
\hypersetup{pdftitle={An Unbounded Gap Between Identification and Identifiable Elicitation: Central Moments},pdfauthor={Jingyi Jia},pdfkeywords={elicitation complexity, identification, central moments}}
\title{An Unbounded Gap Between Identification and Identifiable Elicitation: Central Moments}
\author{Jingyi Jia\\[0.3em]{\normalsize University of Chicago}}
\date{October 11, 2026}
\begin{document}
\maketitle
\begin{abstract}
For $n\geq2$, two expectation equations identify the mean together with the $n$th central
moment. How many coordinates must a report have if a one-observation loss
uniquely elicits it and the report determines that moment? We prove that
$n$ coordinates are necessary and sufficient among all identifiable
intermediaries. The result holds on every open convex set of positive
distributions on any $n+1$ distinct real outcomes. It requires no regularity
of the loss, report map, identification function, or link. Thus
$\elic_{\mathcal I}(\mu_n)=n$ while $\iden(\mu_n)\leq2$, giving an unbounded
gap between identifiable elicitation and indirect identification.
This resolves the central-moment conjecture of Frongillo and Kash on the
stated domains. The proof constructs a $C^1$ manifold of supporting
coefficients from the affine fibers of a convex negative Bayes risk.
Translations along that manifold force an $n$-dimensional tangent space.
Standardized moments of order at least three have the same exact complexity.
The bound also extends
to convex distribution classes with a full-dimensional raw-moment image.
\end{abstract}
\par\noindent\textbf{Keywords.}
elicitation complexity, identification, central moments, Bayes risk, convex analysis
\par

\section{Introduction}
A loss elicits a statistic when its expected value has a unique minimum at
that statistic. When no such loss exists, one can elicit an auxiliary vector
and compute the target from it. Elicitation complexity measures the smallest
possible dimension of that vector. The choice of auxiliary reports matters.
Without restrictions, a real number can encode an entire distribution
\citep{FK2021}. We study identifiable reports, whose values are characterized
by the zeros of expectation equations.

For a real random variable $Y$, the $n$th central moment is
\begin{equation}\label{eq:central}
 \mu_n(P)=\E_P[(Y-\E_PY)^n].
\end{equation}
The first $n$ raw moments give an elicitable and identifiable report for
$\mu_n$. In contrast, the pair $(\E_PY,\mu_n(P))$ has a two-coordinate
identification function. This leaves a basic question. Can a lower-dimensional
identifiable report also arise as the unique minimizer of an expected loss?

\citet[p.~8, immediately after Corollary~2]{FK2015pre} conjectured that the
answer is no below dimension $n$. They also noted that this bound would
separate identification from elicitation by an arbitrarily large amount.
\citet{FW2024} record the central-moment gap among the remaining questions
in elicitation complexity. We prove the conjectured bound on every open
convex subset of the positive simplex over $n+1$ distinct real outcomes.
Writing $\iden$ for indirect identification complexity, our result gives
\begin{equation}\label{eq:headline}
 \elic_{\mathcal I}(\mu_n)=n,\qquad \iden(\mu_n)\leq2,\qquad
 \elic_{\mathcal I}(\mu_n)-\iden(\mu_n)\geq n-2.
\end{equation}
Section~\ref{sec:setting} defines both quantities. The same statement holds
on convex distribution classes whose raw-moment vectors have affine hull
$\R^n$, under the natural integrability condition.
These classes can consist entirely of distributions with densities, so the
result is not restricted to discrete distributions.

The lower bound covers arbitrary jointly elicitable identifiable reports.
Their coordinates need not be separately elicitable or consist of selected
moments. The loss, report map, identification function, and final link need
no continuity, differentiability, or definability. Expected loss must have
a unique attained minimum, and each loss uses one observation.
The gap in \eqref{eq:headline} concerns this decision-theoretic representation.
It is not a bound on sample size, memory, or the dimension used by arbitrary
estimation algorithms.

The proof combines two constraints on reports. Identification makes every
report fiber a full affine flat intersected with the distribution domain.
Unique elicitation makes these fibers the affine pieces of one convex
negative Bayes risk. Fixing the mean and central moment leaves a local
affine flat of dimension $n-2$. Thus a target level set can accommodate a
fiber of codimension two, even when the sharp elicitation bound is $n$.
We must therefore use compatibility across fibers. This lets us construct a $C^1$ manifold
of supporting coefficients even when the original reports have no regularity.
Mean recovery then produces a fixed tangent direction. Local translations
preserve a tangent space while the recovered mean varies over an interval.
The central-moment gradients over that interval span $\R^n$.

\citet[Proposition~5]{FK2015pre} already supply the polynomial argument for
mean recovery. Our contribution is the dimension argument across all report
fibers, with the dual coefficient graph in Lemma~\ref{lem:dualgraph} as its
main step. Earlier conditional-elicitation and vector-valued arguments study
particular intermediate properties under additional assumptions
\citep{FK2015pre,FK2015vector}. They do not rule out every lower-dimensional
identifiable intermediary. Section~\ref{sec:discussion} compares the precise
complexity classes and the one-observation restriction.

Section~\ref{sec:geometry} constructs the dual coefficient graph.
Section~\ref{sec:lowerbound} proves the lower bound from mean recovery and a
Vandermonde span. Section~\ref{sec:consequences} treats standardized moments
and broader distribution classes. Appendices~\ref{app:smooth} and
\ref{app:definable} give complete proofs for two further classes of reports,
without requiring identification.

\section{Setting and main results}\label{sec:setting}
Fix $n\geq2$ and distinct outcomes $y_0,\ldots,y_n\in\R$.
Positive probability vectors on these outcomes form an open simplex.
Use raw moments as its coordinates:
\begin{equation}\label{eq:coordinates}
 x_j=\sum_{i=0}^n p_i y_i^j,\qquad j=1,\ldots,n.
\end{equation}
The Vandermonde matrix is invertible. Thus these coordinates identify the
simplex with an open convex set in $\R^n$. Let $U$ be any nonempty open
convex subset of this moment set. Write $P_x$ for the distribution at $x$.

\begin{definition}\label{def:identifiable}
An identifiable indirect elicitation of $T:U\to\R$ consists of a report
map $\Gamma:U\to\mathcal R\subseteq\R^d$, a finite real loss
$L:\mathcal R\times\{y_0,\ldots,y_n\}\to\R$, and a link
$\psi:\Gamma(U)\to\R$. It satisfies
\begin{equation}\label{eq:identelicitation}
 \{\Gamma(x)\}=\arg\min_{r\in\mathcal R}\E_{P_x}L(r,Y),
 \qquad T(x)=\psi(\Gamma(x)).
\end{equation}
There is also an identification function
$V:\Gamma(U)\times\{y_0,\ldots,y_n\}\to\R^d$ such that
\begin{equation}\label{eq:identifies}
 \E_{P_x}V(r,Y)=0\quad\Longleftrightarrow\quad\Gamma(x)=r
 \qquad(r\in\Gamma(U),\ x\in U).
\end{equation}
The identifiable elicitation complexity $\elic_{\mathcal I}(T)$ is the smallest such $d$.
\end{definition}

The definition requires neither regular report labels nor a regular report
domain. The identification function has $d$ output coordinates but need
not have full rank. No regularity assumption is imposed on $L$, $V$,
$\Gamma$, or $\psi$. Finite support makes all these expectations finite.

\begin{theorem}\label{thm:identifiablecentral}
For every $n\geq2$, the identifiable elicitation complexity of $\mu_n$
on $U$ is exactly $n$.
\end{theorem}

For comparison, the \emph{indirect identification complexity} $\iden(T)$
is the smallest $d$ for which $T=\psi\circ\Gamma$ and
\eqref{eq:identifies} holds for some $\Gamma:U\to\R^d$ and $V$.
This definition does not require a loss that elicits $\Gamma$.
As in Definition~\ref{def:identifiable}, the identification function has
$d$ coordinates, and the link is arbitrary.

\begin{corollary}[An unbounded gap]\label{cor:gap}
For every $n\geq2$ on the domains above,
\[
 \iden(\mu_n)\leq2,\qquad
 \elic_{\mathcal I}(\mu_n)-\iden(\mu_n)\geq n-2.
\]
Consequently the gap is unbounded as the moment order increases.
\end{corollary}
\begin{proof}
The report $(m,c)=(\E_PY,\mu_n(P))$ has identification function
\begin{equation}\label{eq:identification}
 V((m,c),y)=\bigl(y-m,(y-m)^n-c\bigr).
\end{equation}
Its expectation vanishes exactly at that report. Apply
Theorem~\ref{thm:identifiablecentral} to obtain the gap.
\end{proof}

For example, take $n=4$. Two expectation equations identify the mean and fourth central
moment. Yet any identifiable report that determines the fourth central
moment and uniquely minimizes a one-observation expected loss needs at least
four coordinates on $U$.

The matching elicitation upper bound uses the raw-moment report and loss
\begin{equation}\label{eq:upper}
 \mathcal R=\R^n,\qquad L(r,y)=\sum_{j=1}^n(r_j-y^j)^2.
\end{equation}
The unique optimal report is $(x_1,\ldots,x_n)$.
It is identifiable by $V_j(r,y)=y^j-r_j$, and it determines $\mu_n$.
Finite support makes each expected loss finite.

For $n\geq3$, define the standardized central moment
\begin{equation}\label{eq:standardized}
 s_n(x)=\frac{\mu_n(x)}{\mu_2(x)^{n/2}}.
\end{equation}
Every distribution in $U$ has positive variance.

\begin{theorem}\label{thm:identifiablestandardized}
For every $n\geq3$, the identifiable elicitation complexity of $s_n$
on $U$ is exactly $n$.
\end{theorem}

In particular, skewness requires three coordinates and kurtosis requires
four. Subtracting three from kurtosis leaves its complexity unchanged,
since that transformation is invertible. The central-moment lower bound is
local, so it also applies on arbitrarily small open convex neighborhoods.

\section{Geometry of identifiable reports}\label{sec:geometry}
The argument has two parts. First, we construct a smooth manifold in the
space of supporting coefficients, using only identification and unique
elicitation. Second, we use the central-moment polynomial to force that
manifold to have dimension at least $n$. The present section proves the
first part. A fiber means the set of distributions with one fixed report.

The following geometry uses no regularity of the loss or report labels. Write the negative expected
loss and the negative Bayes risk as
\begin{equation}\label{eq:bayesconvex}
 \ell_r(x)=-\E_{P_x}L(r,Y),\qquad f(x)=\sup_r\ell_r(x).
\end{equation}
Each $\ell_r$ is affine. Unique attained minima make $f$ finite on $U$.
Thus $f$ is convex, continuous, and locally Lipschitz.
A subgradient $a\in\partial f(x)$ satisfies
$f(z)\geq f(x)+a\cdot(z-x)$ for every $z\in U$.
The affine-segment observation below is the negative-Bayes-risk form of the strict concavity lemma in
\citet[Lemma 8]{FK2021}.

\begin{lemma}[Full affine fibers]\label{lem:affinefibers}
Under Definition~\ref{def:identifiable}, $f$ is affine on $[x,z]$ if and
only if $\Gamma(x)=\Gamma(z)$. Each report fiber has the form
\begin{equation}\label{eq:fullfiber}
 F_r=U\cap(x_r+K_r),\qquad \dim K_r\geq n-d.
\end{equation}
The full subdifferential $\partial f$ is constant on each fiber.
\end{lemma}

\begin{proof}
A shared report supplies one affine function attaining $f$ at both
endpoints. Convexity forces equality on the segment. Conversely, suppose
$f$ is affine on the segment. The affine loss attained at its midpoint
has nonnegative gaps below $f$ at the two endpoints. Their average is zero,
so both gaps vanish. Uniqueness gives the same report at both endpoints.

For a fixed report, the expected identification function is affine in $x$.
Its zero set is an affine flat with at most $d$ defining equations.
This proves \eqref{eq:fullfiber}.

Let $x,z$ lie in one fiber. Its flat extends the segment slightly beyond
$x$ within $U$. Every subgradient at $x$ must equal the affine slope of
$f$ along that line, since both directions are feasible. It therefore also
supports $f$ at $z$. Reversing the two points proves
$\partial f(x)=\partial f(z)$.
\end{proof}

\begin{lemma}[A continuous family of direction spaces]\label{lem:continuousfibers}
Let $k$ be the minimum fiber dimension in \eqref{eq:fullfiber}.
There is a nonempty open subset $W\subseteq U$ on which every fiber has
dimension $k$. Its direction space $K(x)$ varies continuously on $W$.
\end{lemma}

\begin{proof}
For every $j$, the set of points with fiber dimension at least $j$ is
relatively closed. Indeed, let $x_i\to x$ and choose $j$-planes
$E_i\subseteq K(x_i)$. A subsequence converges to a $j$-plane $E$ by
compactness of the Grassmannian. Inside a fixed small ball about $x$,
the affine identities for $f$ on $x_i+E_i$ pass to the limit by
continuity. Hence $f$ is affine on a disk in $x+E$.
Lemma~\ref{lem:affinefibers} puts that disk in the fiber of $x$.
Its dimension is therefore at least $j$.

The minimum dimension is attained and its locus is open. On this locus,
the same argument puts every limiting $k$-plane inside $K(x)$.
Equality of dimensions identifies the limit, proving continuity.
\end{proof}

The next lemma obtains differentiability in dual coordinates. It does not
assert that the report map or the direction spaces are differentiable.

\begin{lemma}[A dual coefficient graph]\label{lem:dualgraph}
Suppose the fibers have constant dimension $k<n$ near $x_*$. Put
$K_0=K(x_*)$ and $T=K_0^\perp$. Translate the origin to $x_*$.
On a small ball in $T$, write
\begin{equation}\label{eq:leafgraph}
 K(t)=\{(A(t)s,s):s\in K_0\}\subseteq T\oplus K_0,
 \qquad g(t)=f(t,0).
\end{equation}
The map $A$ is continuous and $g$ is strictly convex.
There are an open set $P\subseteq T$, a continuous map $t:P\to T$,
and a $C^1$ map $q:P\to K_0$ such that
\begin{equation}\label{eq:dualgraph}
 a(p)=(p,q(p))\in\partial f(t(p),0),\qquad
 Dq(p)=-A(t(p))^{\mathsf T}.
\end{equation}
Consequently the graph $S=a(P)$ has dimension $n-k$ and
\begin{equation}\label{eq:dualtangent}
 T_{a(p)}S=K(t(p))^\perp.
\end{equation}
One can choose $p_0\in P$ so that $g$ is differentiable at $t(p_0)$
and $\nabla g(t(p_0))=p_0$.
\end{lemma}

\begin{proof}
Continuity of $K$ makes all nearby direction spaces transverse to $T$,
giving the continuous graph map $A$. A fiber through a transversal point
contains $(t+A(t)s,s)$ for all sufficiently small $s$.
Choose a smaller transversal ball whose closure lies in the patch.
Continuity bounds $\|A(t)\|$ on this compact closure, which has positive
distance from the complement of $U$. A uniform
$\delta>0$ therefore makes these points belong to $U$ for every
$\|s\|\leq\delta$ and every $t$ in that ball.

If $g$ had an affine segment between distinct transversal points, those
points would share a fiber by Lemma~\ref{lem:affinefibers}.
Transversality excludes this. Thus $g$ is strictly convex.
Choose an interior differentiability point $t_0$, which exists because a
finite convex function is differentiable almost everywhere. Set
$p_0=\nabla g(t_0)$ and choose a closed transversal ball $C$ centered
at $t_0$ inside the previous ball. The function $g(t)-p_0\cdot t$
has a unique minimum at $t_0$ and a positive gap on $\partial C$.
For $p$ in a neighborhood $P$ of $p_0$, it follows that
\begin{equation}\label{eq:dualminimizer}
 t(p)=\arg\min_{t\in C}\{g(t)-p\cdot t\}
\end{equation}
is unique and interior. Compactness and uniqueness make $t(p)$ continuous,
and $p\in\partial g(t(p))$.

Every subgradient $p$ of the restriction $g$ extends to a subgradient
$(p,q)$ of $f$ by the subgradient restriction rule
\citep[Theorem~23.9]{Rockafellar1970}. Its domain condition holds because
the transversal meets the open set $U$. For a direct proof, the directional
derivative $f'((t,0);v)$ is
finite and sublinear in $v$. The functional $p$ on $T$ lies below its
restriction. Hahn--Banach extends $p$ to a linear functional below
$f'((t,0);\cdot)$ on the full space. For every $z\in U$,
convexity gives $f'((t,0);z-(t,0))\leq f(z)-f(t,0)$.
Thus the extended functional is a subgradient of $f$ at $(t,0)$.
The extension is unique after fixing $p$, since all subgradients agree on
$K(t)$ because $f$ is affine in both directions along the fiber.
If $\eta(t)$ is its affine slope in the variable $s$, then
$A(t)^{\mathsf T}p+q=\eta(t)$, which determines $q$.

Define $q(p)$ by this extension at $t(p)$. Put $t=t(p)$,
$t'=t(p')$, $\Delta p=p'-p$, and $\Delta q=q(p')-q(p)$.
Lemma~\ref{lem:affinefibers} makes $(p,q(p))$ a subgradient at
$(t+A(t)s,s)$ for all $\|s\|\leq\delta$.
The same $\delta$ applies to every transversal point, so this assertion
holds uniformly for both signs of each fiber displacement.
Monotonicity of subgradients, compared with $(t',0)$, gives
\[
 0\leq\Delta p\cdot(t'-t)
       -\bigl(\Delta q+A(t)^{\mathsf T}\Delta p\bigr)\cdot s.
\]
Taking the supremum over that ball in $s$ yields
\begin{equation}\label{eq:dualestimate}
 \delta\|\Delta q+A(t)^{\mathsf T}\Delta p\|
 \leq\Delta p\cdot(t'-t)
 \leq\|\Delta p\|\,\|t'-t\|.
\end{equation}
For $p'\ne p$, divide \eqref{eq:dualestimate} by $\delta\|p'-p\|$.
The resulting upper bound $\delta^{-1}\|t(p')-t(p)\|$ tends to zero
by continuity of $t$. This is the Fr\'echet remainder estimate for
$Dq(p)=-A(t(p))^{\mathsf T}$. Continuity of $A$ makes this derivative continuous.
The tangent vectors $(v,-A(t(p))^{\mathsf T}v)$ are exactly the vectors
orthogonal to $K(t(p))$, proving \eqref{eq:dualtangent}.
\end{proof}

\section{Proof of the lower bound}\label{sec:lowerbound}
In moment coordinates, the target is the polynomial
\begin{equation}\label{eq:polynomial}
 F_n(x)=(1-n)(-x_1)^n+
 \sum_{j=2}^n \binom{n}{j}(-x_1)^{n-j}x_j.
\end{equation}
Its derivative with respect to $x_n$ equals one.

\begin{lemma}[Mean recovery]\label{lem:meanrecovery}
Every elicitable report from which $F_n$ can be computed is constant-mean
on each of its fibers, for $n\geq2$.
\end{lemma}

\begin{proof}
Let $x$ and $x'$ have the same report. Linearity of expected loss and
uniqueness of its minimum imply that every mixture
$x+t(x'-x)$, $0\leq t\leq1$, has that report. Hence $F_n$ is constant
on the segment. Put $h=x'-x$. The coefficient of $t^n$ in
$F_n(x+th)$ is
\begin{equation}\label{eq:leading}
 (1-n)(-h_1)^n.
\end{equation}
Every term in the sum in \eqref{eq:polynomial} has degree at most $n-1$
in $t$. Since $n\geq2$, constancy forces $h_1=0$.
\end{proof}

This is the polynomial observation in Proposition 5 of
\citet{FK2015pre}. We include it to identify exactly where the new
dimension argument starts.

\begin{proofof}{Theorem~\ref{thm:identifiablecentral}}
Suppose $d<n$. Lemma~\ref{lem:affinefibers} gives minimum fiber
dimension $k\geq n-d\geq1$. The mean-recovery argument in
Lemma~\ref{lem:meanrecovery} uses only unique elicitation.
Thus every fiber has constant mean, its direction lies in $e_1^\perp$,
and $k<n$. Take the patch and dual graph above. Then $e_1\in T$ and
$A(t)^{\mathsf T}e_1=0$. The orthogonal coordinates can preserve $e_1$.
Equation~\eqref{eq:dualgraph} gives $Dq(p)e_1=0$.
Shrink $P$ to a rectangular neighborhood of $p_0$ with first axis $e_1$.
Then $q$ is independent of the first coordinate. Its graph and tangent
space are unchanged along
\begin{equation}\label{eq:dualtranslation}
 a(p_0+\tau e_1)=a(p_0)+\tau e_1.
\end{equation}

The mean of $t(p_0+\tau e_1)$ is continuous and is not constant.
Indeed, suppose it equals the mean at $t_0=t(p_0)$ for some
$\tau\ne0$. The two subgradient inequalities for $g$ have total gap
\[
 \tau e_1\cdot\bigl(t(p_0+\tau e_1)-t_0\bigr)=0.
\]
Both must be equalities. Strict convexity forces $t(p_0+\tau e_1)=t_0$.
Differentiability of $g$ at $t_0$ then forces $p_0+\tau e_1=p_0$,
a contradiction. Thus these means contain a nontrivial interval.
This argument allows multiple subgradients at other transversal points.

The central-moment polynomial \eqref{eq:polynomial} is constant on each
fiber, so its gradient annihilates $K(t(p))$. After removing its first
coordinate, \eqref{eq:dualtangent} gives
\[
 e_1,\quad
 b_n(m)=\left(0,
       \left(\binom{n}{j}(-m)^{n-j}\right)_{j=2}^n\right)
 \ \in T_{a(p)}S,
\]
where $m$ is the fiber mean in the original moment coordinates.
Along \eqref{eq:dualtranslation}, the tangent space is fixed and $m$
ranges over an interval. The last $n-1$ coordinates of $b_n(m)$
span $\R^{n-1}$ by the Vandermonde identity. Together with $e_1$,
they force $\dim S\geq n$. But $\dim S=n-k<n$, a contradiction.

The raw-moment loss in \eqref{eq:upper} supplies the matching upper bound.
\label{end:mainproof}
\end{proofof}

The smoothness of the coefficient graph is a conclusion of
\eqref{eq:dualestimate}. The proof does not replace the original reports
with distinct subgradients and claim unique elicitation of those subgradients.
Several subgradients may belong to the same original report fiber.

\section{Consequences and extensions}\label{sec:consequences}
\subsection{Standardized moments}
\begin{lemma}[Reduction to a central moment]\label{lem:standardreduction}
Suppose a uniquely elicitable report determines $s_n$, with $n\geq3$.
On some nonempty open convex ball in $U$, the same report determines $\mu_n$.
\end{lemma}

\begin{proof}
The derivative
\begin{equation}\label{eq:standardderivative}
 \frac{\partial s_n}{\partial x_n}=\mu_2^{-n/2}>0
\end{equation}
shows that $s_n$ is nonconstant on every open moment set.
If $n$ is odd, restrict to a nonempty open convex ball on which $s_n\ne0$.
If $n$ is even, choose such a ball on which
\begin{equation}\label{eq:exceptional}
 s_n\ne(-1)^{n/2+1}(n-1).
\end{equation}
These restrictions are possible by \eqref{eq:standardderivative}.

Take two distributions in this ball with the same report. Every mixture
has that report and the same standardized moment $s$.
Along the mixture line let $A(t)=\mu_n(x+th)$ and
$B(t)=\mu_2(x+th)>0$.
If $n$ is odd, square the defining identity to obtain
\begin{equation}\label{eq:oddidentity}
 A(t)^2=s^2B(t)^n.
\end{equation}
The coefficient of $t^{2n}$ on the left is
$(n-1)^2h_1^{2n}$. On the right it is $-s^2h_1^{2n}$.
Therefore $h_1=0$.
If $n$ is even, use $A(t)=sB(t)^{n/2}$ without squaring.
Its leading coefficients give
\begin{equation}\label{eq:evenleading}
 -(n-1)h_1^n=s(-1)^{n/2}h_1^n.
\end{equation}
The excluded value in \eqref{eq:exceptional} again forces $h_1=0$.

With the mean fixed, both $A(t)$ and $B(t)$ are affine in $t$.
For odd $n$, \eqref{eq:oddidentity} compares a polynomial of degree at
most two with $s^2B(t)^n$. Since $s\ne0$ and $n>2$, $B(t)$ must be
constant. For even $n\geq4$, $s>0$ and $n/2>1$. The identity
$A(t)=sB(t)^{n/2}$ again makes $B(t)$ constant.
In both cases, $A(t)$ is then constant as well.

Thus $\mu_n$ is constant on every report fiber in the restricted ball.
It defines a link from that same report to $\mu_n$.

\end{proof}

\begin{proofof}{Theorem~\ref{thm:identifiablestandardized}}
The raw-moment report in \eqref{eq:upper} gives the upper bound.
For the lower bound, apply Lemma~\ref{lem:standardreduction}.
Restriction to its ball preserves identification and unique elicitation.
Theorem~\ref{thm:identifiablecentral} therefore gives $d\geq n$.
\end{proofof}

\subsection{Convex distribution classes}
The identifiable lower bound also applies to distribution classes that
contain no finitely supported distributions. The relevant requirement is
affine independence of raw-moment vectors.

\begin{corollary}\label{cor:richclass}
Let $\mathcal P$ be a convex class of real-valued distributions with finite
absolute $n$th moments. Suppose the affine hull of
\[
 \{(\E_PY,\ldots,\E_PY^n):P\in\mathcal P\}
\]
is $\R^n$. Then the identifiable elicitation complexity of $\mu_n$ on
$\mathcal P$ is $n$. The same holds for $s_n$ when $n\geq3$ and every
distribution in $\mathcal P$ has positive variance. Loss and identification
coordinates are required to be integrable on $\mathcal P$.
\end{corollary}

\begin{proof}
Choose $n+1$ distributions in $\mathcal P$ with affinely independent
raw-moment vectors. Their positive mixtures have an open simplex of
raw-moment coordinates. For every report, expected loss and expected
identification are affine in those coordinates. The proofs of
Theorems~\ref{thm:identifiablecentral} and
\ref{thm:identifiablestandardized} therefore apply without change.
Their only use of the finite outcome support was to supply these affine
coordinates and finite expectations. The integrability assumption supplies
the latter here. The loss in \eqref{eq:bregman} gives the matching upper bound.
Its raw-moment report is identifiable.
\end{proof}

The upper bound extends to all distributions with a finite absolute $n$th
moment. To avoid the stronger moment requirement of squared loss, use
$\phi(z)=\sqrt{1+z^2}$ and the nonnegative loss
\begin{equation}\label{eq:bregman}
 L(r,y)=\sum_{j=1}^n\left\{\phi(y^j)-\phi(r_j)
             -\phi'(r_j)(y^j-r_j)\right\}.
\end{equation}
Strict convexity of $\phi$ makes the unique minimizer the raw-moment vector.
The risk is finite because $\phi$ grows at most linearly and $\phi'$ is
bounded. Its excess risk in coordinate $j$ is
$\phi(x_j)-\phi(r_j)-\phi'(r_j)(x_j-r_j)$.
This gives the matching upper bound under the natural moment condition.

\section{Related work and discussion}\label{sec:discussion}
\citet{LPS2008} require the intermediate coordinates to be separately
elicitable and continuous. They give complexity three for skewness and a
three-to-four bound for kurtosis, with four conjectured. Our theorem permits
jointly elicitable identifiable intermediaries with no regularity of their
labels. The two intermediary classes are not identical. We therefore do
not claim to settle the entire continuous-coordinate formulation of their
kurtosis problem.

The central-moment conjecture in \citet{FK2015pre} uses the identifiable
intermediary framework formalized in \citet{FK2021}. Its mean-recovery
argument is Lemma~\ref{lem:meanrecovery} here. The affine-segment criterion
in Lemma~\ref{lem:affinefibers} is also known from the strict concavity of
Bayes risk relative to an elicited property \citep[Lemma~8]{FK2021}.
Our proof adds the continuous family of minimal fiber directions and the
dual coefficient graph. These steps control all candidate intermediaries
through one convex function.

Conditional elicitation studies whether a property becomes elicitable
when another property is known. \citet{FK2015pre} use this approach to
exclude selected incomplete moment vectors under differentiability
assumptions. \citet{FK2015vector} develop vector-valued characterizations
that also constrain specific properties. Our argument requires no Hessian
and constructs the needed differentiability in supporting-coefficient
coordinates. It does not assume differentiability of the original report.

The number of observations in the loss changes the problem.
\citet{CFMW2017} study losses that evaluate a report against several
independent observations. They give central-moment tradeoffs between report
dimension and observation count. The bound here fixes the observation count
at one. It places no restriction on the dimensions possible with
multi-observation losses.

\citet{FKB2016} also ask for lower bounds under regularity assumptions.
Appendix~\ref{app:smooth} proves the same exact bounds when the loss and
report map are $C^1$, without identification. Appendix~\ref{app:definable}
replaces those smoothness assumptions by definability in one o-minimal
expansion of the real field. That result includes nonsmooth semialgebraic
losses and discontinuous optimal-report maps. These statements are separate
from the unrestricted identifiable theorem.

The dimension restriction in the main theorem comes from the number of
identification equations. It does not come from the numerical labels of
the reports. Removing identification and both alternative regularity
conditions admits arbitrary encodings of a distribution into one real
report \citep{FK2021}. Losses with nonunique minimizers also fall outside
our result. Extending the argument to such losses would require control of
contact sets that need not coincide with single report fibers.

\section*{AI Disclosure}
OpenAI Codex assisted with problem formulation, mathematical argument
development, literature searches, drafting, and proof checking.
ChatGPT assisted with proof review.
The author reviewed the complete manuscript, including the mathematical
arguments and proofs.

\bibliography{references}

@misc{FK2015pre,
  author = {Rafael Frongillo and Ian A. Kash},
  title = {On Elicitation Complexity and Conditional Elicitation},
  year = {2015},
  eprint = {1506.07212},
  archivePrefix = {arXiv},
  primaryClass = {cs.LG},
  url = {https://arxiv.org/abs/1506.07212v1},
  note = {Version 1. The central-moment conjecture appears on page 8 after Corollary 2}
}

@article{FK2021,
  author = {Frongillo, Rafael M and Kash, Ian A},
  title = {Elicitation complexity of statistical properties},
  journal = {Biometrika},
  volume = {108},
  number = {4},
  pages = {857--879},
  year = {2021},
  doi = {10.1093/biomet/asaa093}
}

@inproceedings{FKB2016,
  author = {Rafael Frongillo and Ian A. Kash and Stephen Becker},
  title = {Open Problem: Property Elicitation and Elicitation Complexity},
  booktitle = {29th Annual Conference on Learning Theory},
  series = {Proceedings of Machine Learning Research},
  volume = {49},
  pages = {1655--1658},
  publisher = {PMLR},
  year = {2016},
  url = {https://proceedings.mlr.press/v49/frongillo16.html}
}

@article{FW2024,
  author = {Frongillo, Rafael and Waggoner, Bo},
  title = {Recent Trends in Information Elicitation},
  journal = {ACM SIGecom Exchanges},
  volume = {22},
  number = {1},
  pages = {122--134},
  year = {2024},
  doi = {10.1145/3699824.3699836}
}

@inproceedings{LPS2008,
  author = {Lambert, Nicolas S. and Pennock, David M. and Shoham, Yoav},
  title = {Eliciting properties of probability distributions},
  booktitle = {Proceedings of the 9th ACM conference on Electronic commerce},
  pages = {129--138},
  year = {2008},
  doi = {10.1145/1386790.1386813}
}

@article{vdD1998,
  author = {van den Dries, Lou},
  title = {O-minimal Structures and Real Analytic Geometry},
  journal = {Current Developments in Mathematics},
  volume = {1998},
  number = {1},
  pages = {105--152},
  year = {1998},
  doi = {10.4310/cdm.1998.v1998.n1.a4}
}

@inproceedings{FK2015vector,
  author = {Rafael Frongillo and Ian A. Kash},
  title = {Vector-Valued Property Elicitation},
  booktitle = {Proceedings of The 28th Conference on Learning Theory},
  series = {Proceedings of Machine Learning Research},
  volume = {40},
  pages = {710--727},
  publisher = {PMLR},
  year = {2015},
  url = {https://proceedings.mlr.press/v40/Frongillo15.html}
}

@inproceedings{CFMW2017,
  author = {Casalaina-Martin, Sebastian and Frongillo, Rafael and Morgan, Tom and Waggoner, Bo},
  title = {Multi-Observation Elicitation},
  booktitle = {Proceedings of the 2017 Conference on Learning Theory},
  series = {Proceedings of Machine Learning Research},
  volume = {65},
  pages = {449--464},
  publisher = {PMLR},
  year = {2017},
  url = {https://proceedings.mlr.press/v65/casalaina-martin17a.html}
}

@book{Rockafellar1970,
  author = {R. Tyrrell Rockafellar},
  title = {Convex Analysis},
  publisher = {Princeton University Press},
  year = {1970},
  url = {https://www.jstor.org/stable/j.ctt14bs1ff}
}
\appendix
\section{Smooth losses and report maps}\label{app:smooth}
These results do not assume that the intermediary is identifiable.
\begin{definition}\label{def:smooth}
A smooth indirect elicitation of a target $T:U\to\R$ uses an open
report set $\mathcal R\subseteq\R^d$. It consists of a loss $L:\mathcal R\times
\{y_0,\ldots,y_n\}\to\R$, a report map $\Gamma:U\to\mathcal R$, and a
link $\psi:\Gamma(U)\to\R$. Each $L(\cdot,y_i)$ and $\Gamma$ is $C^1$,
and, for every $x\in U$,
\begin{equation}\label{eq:elicitation}
 \{\Gamma(x)\}=\arg\min_{r\in\mathcal R}\E_{P_x}L(r,Y),
 \qquad T(x)=\psi(\Gamma(x)).
\end{equation}
The smooth elicitation complexity is the smallest such $d$.
\end{definition}

The definition requires an attained, unique minimum. It does not assume
that the loss is convex, that its Hessian exists, or that the report is
identifiable. The loss and the report map have separate $C^1$ assumptions.
Smoothness of one does not substitute for smoothness of the other.

\begin{theorem}\label{thm:central}
For every $n\geq2$, the smooth elicitation complexity of $\mu_n$ on $U$
is exactly $n$.
\end{theorem}

The lower bound is local. It holds even when $U$ is an arbitrarily small
open convex neighborhood of one positive probability vector. It therefore
also applies to a larger distribution class whenever a proposed procedure
restricts to Definition~\ref{def:smooth} on such a neighborhood.

\begin{theorem}\label{thm:standardized}
For every $n\geq3$, the smooth elicitation complexity of $s_n$ on $U$
is exactly $n$.
\end{theorem}

In particular, skewness requires three report coordinates and kurtosis
requires four. Subtracting three from kurtosis does not change its
complexity, because that transformation is invertible.

\subsection{The coefficient manifold}
In the smooth setting of Definition~\ref{def:smooth}, every expected loss
on the fixed finite outcome set has the form
\begin{equation}\label{eq:affine}
 R(r,x):=\E_{P_x}L(r,Y)=c(r)+a(r)\cdot x,
\end{equation}
where $c:\mathcal R\to\R$ and $a:\mathcal R\to\R^n$ are $C^1$.
Indeed, each probability $p_i$ is an affine function of $x$.
The vector $a(r)$ records all coefficients except the intercept.

We first isolate the geometry needed below. The tangent space $T_z S$
consists of velocities of differentiable curves through $z$ in a manifold
$S$. All manifolds in the proof are local embedded $C^1$ manifolds.

\begin{lemma}[A regular coefficient patch]\label{lem:patch}
Assume Definition~\ref{def:smooth}, and suppose $\Gamma$ is nonconstant on every
open subset of $U$. There are an integer $1\leq k\leq d$, an open set
$V\subseteq U$, a $k$-dimensional report manifold $M=\Gamma(V)$, and a
$k$-dimensional coefficient manifold $S=a(M)$ with the following properties.
The restrictions $\Gamma:V\to M$ and $a:M\to S$ are respectively a
submersion and a $C^1$ diffeomorphism. For $x\in V$,
\begin{equation}\label{eq:normal}
 T_{a(\Gamma(x))}S=(\ker D\Gamma(x))^\perp.
\end{equation}
If a $C^1$ function $f$ on $V$ is constant on each report fiber, it factors
through a $C^1$ function on $M$ after shrinking the patch.
\end{lemma}

\begin{proof}
Let $k$ be the maximum rank of $D\Gamma$ on $U$. If $k=0$, $\Gamma$ is
locally constant, contrary to the assumption. Choose a point of rank $k$.
A nonzero minor and maximality of $k$ give a neighborhood of constant rank.
The constant-rank theorem supplies a local image manifold $M$ and a
submersion onto $M$.

The map $a$ is injective on attainable reports. To see this, suppose
$a(r)=a(s)$ and both reports are attained. If $c(r)<c(s)$, then $s$ cannot
minimize any risk. The reverse inequality excludes $r$. Equality makes
their risks identical and contradicts uniqueness unless $r=s$.

A $C^1$ injective map on a $k$-manifold has rank $k$ somewhere in every
nonempty open patch. Otherwise its maximal rank on that patch is some
$h<k$. Near a point of rank $h$, the constant-rank theorem gives a
positive-dimensional fiber, contradicting injectivity. Apply this fact to
$a|_M$. Restrict to a patch where $a|_M$ is an embedding. It gives the
claimed $S$ and diffeomorphism.

Fix $x\in V$, let $r=\Gamma(x)$, and use local coordinates $u\in\R^k$
on $M$. Optimality along $M$ gives
\begin{equation}\label{eq:foc}
 D_u c(r)+D_u a(r)^{\mathsf T}x=0.
\end{equation}
The same equation holds for every nearby point in the fiber of $r$.
Every $v\in\ker D\Gamma(x)$ is the velocity of a $C^1$ curve in that
fiber. Differentiating \eqref{eq:foc} along the curve yields
$D_u a(r)^{\mathsf T}v=0$. Thus $T_{a(r)}S$ is contained in
$(\ker D\Gamma(x))^\perp$. Both spaces have dimension $k$, so they agree.
Finally, a submersion has local $C^1$ sections. Compose $f$ with a section
to obtain its factor on $M$. Constancy on fibers makes that factor
well-defined.
\end{proof}

The rank argument removes redundant reports before any dimension bound.
It does not infer a positive-definite Hessian from strict optimality.

\begin{lemma}[Constant tangent directions]\label{lem:translation}
Let $S\subseteq\R^N$ be an embedded $C^1$ manifold. Suppose a fixed vector
$v$ belongs to $T_zS$ at every point of a neighborhood in $S$. Near any
point $z_0$ in that neighborhood, translations $z\mapsto z+tv$ preserve
$S$ for small $t$. Moreover,
\begin{equation}\label{eq:fixedtangent}
 T_{z_0+tv}S=T_{z_0}S.
\end{equation}
\end{lemma}

\begin{proof}
Represent $S$ locally as the graph of a $C^1$ map $g:E\to E^\perp$,
where $E=T_{z_0}S$. Since $v\in E$, tangency at each graph point means
$Dg(u)v=0$. On a sufficiently small ball in $E$, integration along the
$v$ direction gives $g(u+tv)=g(u)$. Differentiate this identity with
respect to $u$ to obtain $Dg(u+tv)=Dg(u)$. The graph and its tangent
space therefore remain unchanged under the local translations.
\end{proof}

The next observation relates these translations to an elicited expectation.
It uses strict optimality, not second derivatives.

\begin{lemma}[Means change along coefficient translations]\label{lem:meanflow}
Suppose the mean $x_1$ is constant on report fibers in the patch of
Lemma~\ref{lem:patch}. Let $m:S\to\R$ denote this recovered mean.
Then $m$ is $C^1$. If $a_0+te_1\in S$ for an interval of $t$, the
function $t\mapsto m(a_0+te_1)$ is strictly decreasing.
\end{lemma}

\begin{proof}
The factorization assertion in Lemma~\ref{lem:patch}, followed by
$(a|_M)^{-1}$, gives the $C^1$ function $m$.
Take $t_2>t_1$, let $r_i$ be the report at coefficient $a_0+t_i e_1$,
and choose $x_i$ with $\Gamma(x_i)=r_i$.
Write $\delta=t_2-t_1>0$ and $\Delta c=c(r_2)-c(r_1)$.
The reports differ. Strict optimality at $x_1$ and $x_2$ gives
\begin{equation}\label{eq:strict}
 \Delta c+\delta(x_1)_1>0,
 \qquad
 \Delta c+\delta(x_2)_1<0.
\end{equation}
Subtracting yields $(x_2)_1<(x_1)_1$, as required.
\end{proof}

\subsection{Central and standardized moments}

\begin{proofof}{Theorem~\ref{thm:central}}
Assume a smooth indirect elicitation in dimension $d$ exists.
Since $\partial F_n/\partial x_n=1$, $\Gamma$ is nonconstant on every
open set. Take the patch in Lemma~\ref{lem:patch}.
Mean recovery implies
\begin{equation}\label{eq:eone}
 e_1\in T_{a(\Gamma(x))}S.
\end{equation}
Also, $F_n$ is constant on each report fiber, so its gradient is
orthogonal to $\ker D\Gamma(x)$. Remove its first coordinate by using
\eqref{eq:eone}. Equation~\eqref{eq:normal} then gives
\begin{equation}\label{eq:btangent}
 b_n(x_1):=\left(0,
       \left(\binom{n}{j}(-x_1)^{n-j}\right)_{j=2}^n\right)
 \in T_{a(\Gamma(x))}S.
\end{equation}
For $n=2$, the expression is simply $b_2(x_1)=e_2$.

Fix $a_0\in S$. Lemma~\ref{lem:translation} gives an interval
$a_0+te_1\in S$ with one fixed tangent space $E=T_{a_0}S$.
By Lemma~\ref{lem:meanflow}, the recovered mean varies continuously
and strictly along this interval. Its values therefore contain an open
interval $I\subseteq\R$. Equation~\eqref{eq:btangent} implies
\begin{equation}\label{eq:span}
 e_1\in E,\qquad b_n(m)\in E\quad\text{for every }m\in I.
\end{equation}
The last $n-1$ coordinates of $b_n(m)$ are nonzero multiples of the
monomials $m^{n-2},m^{n-3},\ldots,1$. Their vectors span $\R^{n-1}$
as $m$ ranges over any open interval. For example, evaluation at $n-1$
distinct points gives an invertible Vandermonde matrix. Thus
\eqref{eq:span} forces $E=\R^n$ and $n=k\leq d$.

The loss in \eqref{eq:upper} supplies the upper bound.
Its raw-moment report is smooth.
\end{proofof}

\begin{proofof}{Theorem~\ref{thm:standardized}}
The upper bound follows from \eqref{eq:upper}.
For the lower bound, Lemma~\ref{lem:standardreduction} gives an open convex
ball on which the same report determines $\mu_n$.
The restricted loss and report map remain $C^1$.
Theorem~\ref{thm:central} therefore gives $d\geq n$.
\end{proofof}

\subsection{Why the report-map assumption is separate}
Smoothness of the loss alone does not imply a smooth or even continuous
minimizer map. For example, let $Y\in\{-1,1\}$ have mean
$m\in(-1/4,1/4)$, and let
$\mathcal R=(-1,1)\cup(-\infty,-2)\cup(2,\infty)$. Define
\begin{equation}\label{eq:discontinuous}
 L(r,y)=\begin{cases}
 r^2,& |r|<1,\\
 r^{-2}-2r^{-1}y,& |r|>2.
 \end{cases}
\end{equation}
This loss is smooth on its open report domain. Its unique optimal report
is $1/m$ when $m\ne0$ and zero when $m=0$. The map is discontinuous
at zero. This example explains why the two $C^1$ assumptions are separate.
It is semialgebraic, so it is within the scope of the definable theorem
when that theorem's target and support conditions hold.

\section{Definable losses and report sets}\label{app:definable}
The next result uses a precise class of sets and functions that admit
finite smooth decompositions. A semialgebraic set is a finite Boolean
combination of polynomial equalities and inequalities. A function is
semialgebraic when its graph is semialgebraic. More generally, fix an
o-minimal expansion of the real field. Its definable sets are closed under
Boolean operations, Cartesian products, and coordinate projections.
Every definable subset of the line is a finite union of points and intervals.
Real parameters are allowed. A function is definable if its graph is.
The semialgebraic sets give one such structure.

\begin{theorem}\label{thm:definable}
Let $\mathcal R\subseteq\R^d$ be a definable report set, and let every
finite-valued loss coordinate $L(\cdot,y_i)$ be definable in the same
o-minimal structure. Suppose the expected loss has a unique attained
minimizer for each distribution in $U$. If these minimizers determine
$\mu_n$, where $n\geq2$, or $s_n$, where $n\geq3$, then $d\geq n$.
The bound is attained by a semialgebraic loss.
\end{theorem}

The report set need not be open or connected. Neither the loss nor the
minimizer map needs to be continuous. The final link remains arbitrary.
This theorem covers polynomial and rational formulas on their domains,
and finite piecewise polynomial formulas. It also covers finite combinations
of absolute values, maxima, and minima. It applies in
other o-minimal expansions when the loss and report set are definable there.

We use smooth cell decomposition in a fixed o-minimal expansion of the
real field \citep{vdD1998}. A definable set has a finite partition into
$C^1$ cells. Given finitely many definable functions, one can choose the
partition so that their restrictions to each cell are $C^1$.
A definable map on an open set is therefore $C^1$ on some nonempty open
cell. A definable subset of $\R^n$ has dimension $n$ exactly when it has
nonempty interior, and dimension respects finite unions.

\begin{proofof}{Theorem~\ref{thm:definable}}
Restrict $U$ to a nonempty open ball whose closure lies in $U$.
This ball is semialgebraic, hence definable even if the original $U$ is not.
Expected risk is a definable function of $(r,x)$ by
\eqref{eq:affine}. The minimizer graph is the definable set
\begin{equation}\label{eq:defgraph}
 \{(x,r):x\in U,\ r\in\mathcal R,\
       R(r,x)\leq R(s,x)\text{ for every }s\in\mathcal R\}.
\end{equation}
Existence and uniqueness make this the graph of a definable map $\Gamma$.

Partition $\mathcal R$ into finitely many $C^1$ cells $Q$ on which all
loss coordinates are $C^1$. The definable sets $\Gamma^{-1}(Q)$ cover
the open ball. At least one has nonempty interior. Inside that interior,
choose a nonempty open cell $C$ on which $\Gamma$ is $C^1$.
Fix $x_0\in C$, and take a local $C^1$ chart
$\chi:O\subseteq\R^s\to Q$ around $\Gamma(x_0)$, with $s=\dim Q\leq d$.
Continuity of $\Gamma$ on $C$ permits a further restriction to an open
convex ball $B\subseteq C$ with $\Gamma(B)\subseteq\chi(O)$.

On $B$, the report map $\chi^{-1}\circ\Gamma$ and the loss coordinates
$L(\chi(\cdot),y_i)$ are $C^1$. Their report domain $O$ is open in
$\R^s$. The original unique global minimizer remains the unique minimizer
over $\chi(O)$, since it lies there. Thus they satisfy
Definition~\ref{def:smooth} on $B$. Theorem~\ref{thm:central} or
Theorem~\ref{thm:standardized} gives $n\leq s\leq d$.
The loss in \eqref{eq:upper} attains the bound and is semialgebraic.
\end{proofof}

\end{document}